\documentclass[reqno,12pt]{amsart}
\usepackage{amsfonts}
\usepackage{bbm}
\usepackage{} %leqno is the option to put formula numbers on the left side
\numberwithin{equation}{section}
\usepackage{indentfirst}
 \usepackage{color}
\usepackage{amssymb}
\usepackage{mathrsfs}
\usepackage{xy}
\xyoption{all}
\def\Ext{\mbox{\rm Ext}\,} \def\Hom{\mbox{\rm Hom}} \def\dim{\mbox{\rm dim}\,} \def\Iso{\mbox{\rm Iso}\,}\def\ra{\rightarrow}
\def\lr#1{\langle #1\rangle}    
\def\Ker{\mbox{\rm Ker}\,}   \def\im{\mbox{\rm Im}\,} \def\Coker{\mbox{\rm Coker}\,}
\def\tw{\mbox{\rm tw}\,}\def\id{\mbox{\rm id}\,}

\def\M{\mathcal{M}}
\def\Aut{\mbox{\rm Aut}\,}\def\Dim{\mbox{\rm \textbf{dim}}\,}\def\A{\mathcal{A}\,} \def\H{\mathcal{H}\,}
\def\P{\mathscr{P}\,}

\theoremstyle{plain} %text of this environment is typesetted in italics
\newtheorem{theorem}{\bf Theorem}[section]
\newtheorem{lemma}[theorem]{\bf Lemma}
\newtheorem{corollary}[theorem]{\bf Corollary}
\newtheorem{proposition}[theorem]{\bf Proposition}

\theoremstyle{definition} %text of this environment is typesetted in roman letters
\newtheorem{definition}[theorem]{\bf Definition}
\newtheorem{remark}[theorem]{\bf Remark}
\newtheorem{example}[theorem]{\bf Example}

\newcommand{\bt}{\begin{theorem}}
\newcommand{\et}{\end{theorem}}
\newcommand{\bl}{\begin{lemma}}
\newcommand{\el}{\end{lemma}}
\newcommand{\bd}{\begin{definition}}
\newcommand{\ed}{\end{definition}}
\newcommand{\bc}{\begin{corollary}}
\newcommand{\ec}{\end{corollary}}
\newcommand{\bp}{\begin{proof}}
\newcommand{\ep}{\end{proof}}
\newcommand{\bx}{\begin{example}}
\newcommand{\ex}{\end{example}}
\newcommand{\br}{\begin{remark}}
\newcommand{\er}{\end{remark}}
\newcommand{\be}{\begin{equation}}
\newcommand{\ee}{\end{equation}}
\newcommand{\ba}{\begin{align}}
\newcommand{\ea}{\end{align}}
\newcommand{\bn}{\begin{enumerate}}
\newcommand{\en}{\end{enumerate}}
\newcommand{\bcs}{\begin{cases}}
\newcommand{\ecs}{\end{cases}}

\makeatletter
\renewcommand{\section}{\@startsection{section}{1}{0mm}
  {-\baselineskip}{0.5\baselineskip}{\bf\leftline}}
\makeatother

\begin{document}

\title[Hall algebras associated to complexes of fixed size]{Hall algebras associated to complexes of fixed size} %title of paper and the running head option

\author[Haicheng Zhang]{{Haicheng Zhang}} %first author's name and the running head option
\address{Institute of Mathematics, School of Mathematical Sciences, Nanjing Normal University,
Nanjing 210023, P.~R.~China}
\email{zhanghai14@mails.tsinghua.edu.cn (H. Zhang)}
%\dedicatory{Dedicated to Professor Xxx Yyy on his sixtieth birthday}

%%%%%%%%%%%%%%% footnote %%%%%%%%%%%%%%%%
%In case \subjclass[2010] command is not effective
%(or the version of amsart.cls is old), write as follows:
%\renewcommand{\thefootnote}{\fnsymbol{footnote}}
%\footnote[0]{2010\textit{ Mathematics Subject Classification}.
%Primary 00; Secondary 00.}
%
%\thanks{ %acknowledgment of support etc. if any
%$^{*}$Thanks.
%}
%%%%%%%%%%%% Authors addresses %%%%%%%%%%%%%

%%%%%%%%%%%%%%%%%%%%%%%%%%%%%%%%%%%%%%%%%

\begin{abstract}
Let $\A$ be a finitary hereditary abelian category with enough projectives. We study the Hall algebra of complexes of fixed size over projectives. Explicitly, we first give a relation between Hall algebras of complexes of fixed size and cyclic complexes. Secondly, we characterize the Hall algebra of complexes of fixed size by generators and relations, and relate it to the derived Hall algebra of $\A$. Finally, we give the integration map on the Hall algebra of $2$-term complexes over projectives.
\end{abstract}

\maketitle

\section{Introduction}
The Hall algebra of a finite dimensional algebra $A$ over a finite field $k$ was introduced by Ringel \cite{R90a} in 1990. Ringel \cite{R90,R90a} proved that if $A$ is representation-finite and hereditary, the Ringel--Hall algebra of $A$ is isomorphic to the positive part of the corresponding quantized enveloping algebra. In order to give an intrinsic realization of the full quantized enveloping algebra via Hall algebra approach, one has managed to define the Hall algebra of a triangulated category satisfying some homological finiteness conditions (cf. \cite{Toen2006}, \cite{XiaoXu}). However, the root category of a finite dimensional algebra does not satisfy the homological finiteness conditions. In other word, the Hall algebra of a root category has not been defined.

In 2013, for each hereditary algebra $A$, Bridgeland \cite{Br} introduced an algebra, called the Bridgeland Hall algebra of \emph{A}, which is the Hall algebra of $2$-cyclic complexes over projective $A$-modules with some localization and reduction. He proved that the quantized enveloping algebra associated to \emph{A} is embedded into the Bridgeland Hall algebra of $A$. This provides a beautiful realization of the entire quantized enveloping algebra by Hall algebras.

Inspired by Bridgeland's work, for each hereditary algebra $A$ and any nonnegative integer $m\neq1$, Chen and Deng \cite{ChenD} applied Bridgeland's construction to $m$-cyclic complexes over projective $A$-modules, and introduced the Bridgeland Hall algebra of $m$-cyclic complexes of $A$, whose algebra structure was characterized in \cite{ZHC2}.

Cluster algebras were introduced by Fomin and Zelevinsky in \cite{FZ} and later the quantum cluster algebras were introduced by Berenstein and Zelevinsky in \cite{BZ05}.
In \cite{DXZ}, the author and his coauthors consider the Hall algebra of 2-term complexes over projective $kQ$-modules of a finite acyclic quiver $Q$, and realize the quantum cluster algebra with principal coefficients as a quotient algebra of the achieved algebra with some localization and twist.

In this paper, let $\A$ be a finitary abelian category with enough projectives and $m$ be a positive integer. The purpose of this paper is to generalize the construction of Hall algebra of 2-term complexes over projectives to $m$-term complexes, and give a characterization on the algebra structure of the achieved algebra. In Section 2 we recall some homological properties of $m$-cyclic complexes and $m$-term complexes over projectives. We establish a relation between Hall algebras of $m$-cyclic complexes and $m$-term complexes in Section 3. From Section 4 on, we assume that $\A$ is hereditary. We characterize indecomposable objects in the categories of $m$-cyclic complexes and $m$-term complexes over projectives in Section 4. In order to study the Hall algebra of $m$-term complexes over projectives, in Section 5 we first give a characterization on the Hall algebra of bounded complexes over projectives, and relate it to the derived Hall algebra of $\A$. Section 6 is devoting to characterizing the Hall algebra of $m$-term complexes over projectives by generators and relations. As an additional result, we give the integration map on the Hall algebra of $2$-term complexes over projectives in Section 7.

Let us fix some notations used throughout the paper. Let $k=\mathbb{F}_q$ be always a finite field with
$q$ elements. Let $\A$ be an (essentially small) finitary abelian $k$-category with enough projectives, and $\mathscr{P}\subset\A$ be the subcategory consisting of projective objects. Given an exact category $\mathcal {E}$, we denote by $C^b(\mathcal {E})$, $K^b(\mathcal {E})$ and $D^b(\mathcal {E})$ the category of bounded complexes over $\mathcal {E}$, the bounded homotopy category, and the bounded derived category, respectively. The Grothendieck group of $\mathcal {E}$ and the set of isomorphism classes $[X]$ of objects in $\mathcal {E}$ are denoted by $K(\mathcal {E})$ and $\Iso(\mathcal {E})$, respectively. For any object $M\in\mathcal {E}$ we denote by $\hat{M}$ the image of $M$ in $K(\mathcal {E})$. For a finite set $S$, we denote by $|S|$ its cardinality. For an object $M$ in an additive category, we denote by $\Aut(M)$ the automorphism group of $M$, and set $a_M:=|\Aut(M)|$.

\section{Cyclic complexes and complexes of fixed size}
In this section, we summarize some necessary homological properties of cyclic complexes and complexes of fixed size. We focus our attention on the complexes over projectives.

\subsection{Cyclic complexes}
For each positive integer $m$, write $\mathbb{Z}_m=\mathbb{Z}/m\mathbb{Z}=\{0,1,\cdots,m-1\}$. By definition,
an \emph{$m$-cyclic complex} ${{M}_{\circ}}={{({{M}_{i}},{{d}_{i}})}_{i\in {{\mathbb{Z}}_{m}}}}$ over $\mathcal{A}$ consists of objects $M_i$ in $\A$ and morphisms $d_i:M_i\rightarrow M_{i+1}$ for $i\in \mathbb{Z}_m$
satisfying $d_{i+1}d_i=0$.
%Hence, each $m$-cyclic complex ${{M}_{\circ }}={{({{M}_{i}},{{d}_{i}})}_{i\in {{\mathbb{Z}}_{m}}}}$ can be diagrammed by $$\xymatrix@!=0.05cm @M=0pt{&&M_0\ar@/^/[dr]^-{d_0}&\\
%&M_{m-1}\ar@/^/[ur]^--{d_{m-1}}&&M_1\ar@/^0.5pc/[dl]\\
%&&\ar@/^/[ul]\cdots &}$$ with $d_{i+1}d_i=0$ for all $i\in \mathbb{Z}_m$.
A \emph{morphism} $f$ between two $m$-cyclic complexes ${{M}_{\circ }}={{({{M}_{i}},{{d}_{i}})}_{i\in {{\mathbb{Z}}_{m}}}}$ and
${{N}_{\circ }}={{({{N}_{i}},{{c}_{i}})}_{i\in {{\mathbb{Z}}_{m}}}}$ is given by a family of morphisms $f_i:M_i\rightarrow N_i$ satisfying $f_{i+1}d_i=c_if_i$
for all $i\in\mathbb{Z}_m$.
%Let $f=(f_i)_{i\in\mathbb{Z}_m}$ and $g=(g_i)_{i\in\mathbb{Z}_m}$ be two morphisms between $m$-cyclic complexes ${{M}_{\bullet }}={{({{M}_{i}},{{d}_{i}})}_{i\in {{\mathbb{Z}}_{m}}}}$ and
%${{N}_{\bullet }}={{({{N}_{i}},{{c}_{i}})}_{i\in {{\mathbb{Z}}_{m}}}}$,
%we say that $f$ is \emph{homotopic} to $g$ if there exist $m$ morphisms $s_i:M_i\rightarrow N_{i-1}$ in $\mathcal{A}$ such that $f_i-g_i=s_{i+1}d_i+c_{i-1}s_i$ for all $i\in\mathbb{Z}_m$.
The category of $m$-cyclic complexes over $\mathcal{A}$ is denoted by ${C}_m(\mathcal{A})$. For notational simplicity, we write $C_0(\A)$ for the category $C^b(\A)$ of bounded complexes over $\A$, and set $\mathbb{Z}_0=\mathbb{Z}$. The bounded complexes are called \emph{$0$-cyclic complexes}.
%and $K_m(\mathcal{A})$ denotes the homotopy category of~${C}_m(\mathcal{A})$ by identifying homotopic morphisms.
For each integer $t$, we have a shift functor $$[t]:\mathcal{C}_m(\mathcal{A})\rightarrow \mathcal{C}_m(\mathcal{A}),~~M_\circ \mapsto M_\circ[t],$$
where $M_\circ[t]=(X_i,f_i)$ is defined by $$X_i=M_{i+t},~~f_i=(-1)^{t}d_{i+t},~~ i\in {\mathbb{Z}_m}.$$

For $m\geq0$, let $C_m(\P)$ be the subcategory of $C_m(\A)$, which is consisting of $m$-cyclic complexes over $\P$. In the sense of component-wise exactness, $C_m(\P)$ is closed under extensions.
For any morphism $f:Q\ra P$~of projectives, if $m\neq1$ one defines $C_f=(M_i,d_i)\in \mathcal{C}_m(\P)$ by
$$M_i=\begin{cases} Q\;\;&\text{~$i=m-1$};\\
                    P\;\;&\text{~$i=0$};\\
                     0 &\text{otherwise,}\end{cases}\qquad
 d_i=\begin{cases} f\;\;&\text{~$i=m-1$};\\
                     0 &\text{otherwise.}\end{cases}$$
If $m=1$ one defines $$C_f={\left(P\oplus Q,\begin{pmatrix}0&f\\0&0\end{pmatrix}\right)}\in C_1(\mathscr{P}).$$
So each projective object $P$ determines an object ${{K}_{P}}:=C_{\id_P}$ in $\mathcal{C}_m(\mathscr{P})$.
By \cite[Lemma 2.3]{ChenD}, all indecomposable projective (injective) objects in $C_m(\P)$ are of the form $K_P[r]$ for some indecomposable object $P\in\P$ and $r\in\mathbb{Z}_m$. That is, $C_m(\P)$ is a Frobenius exact category.

\subsection{Complexes of fixed size}
For each positive integer $m$,
we consider the category $C^m(\A)$, it is the full subcategory of $C^b(\A)$ whose objects are the complexes $M_\bullet=(M_i,d_i)$ with $M_i=0$ if $i\notin\{1,2,\cdots,m\}.$ That is, $$M_{\bullet}=\xymatrix{M_1\ar[r]^-{d_1}&M_2\ar[r]^-{d_2}&\cdots \ar[r]&M_{m-1}\ar[r]^-{d_{m-1}}&M_{m}}$$ with $d_{i+1}d_i=0$ for $1\leq i<m-1$. Each object in $C^m(\A)$ is called an {\em $m$-term complex} over $\A$.
Let $C^m(\P)$ be the subcategory of $C^m(\A)$, which is consisting of $m$-term complexes over $\P$.
Clearly, $C^m(\A)$ is an abelian category, and $C^m(\P)$ is closed under extensions. Actually, $C^1(\P)=\P$.
For $m\geq2$, the Auslander--Reiten theory and some homological properties of $C^m(\P)$ are studied in \cite{Bau2,Cha,Cha2}. In fact, it is proved in \cite{Bau2} that $C^m(\P)$ is an exact category with enough projectives and injectives, and its global dimension is $m-1$. Following \cite{Bau2}, given an object $P\in\P$, we consider the following objects in $C^m(\P)$:
\begin{itemize}
\item[$\bullet$] $J_P=(M_i,d_i)$ with $M_i=0$ if $i\notin\{m-1,m\}$, $M_{m-1}=M_{m}=P$ and $d_{m-1}=\id_P;$
\item[$\bullet$] $S_P=(M_i,d_i)$ with $M_i=0$ if $i\neq1$, and $M_1=P$;
\item[$\bullet$] $T_P=(M_i,d_i)$ with $M_i=0$ if $i\neq m$, and $M_{m}=P$.
\end{itemize}
By \cite[Corollary 3.9]{Bau2}, all indecomposable projective objects in $C^m(\P)$ are of the form $T_P$ for some indecomposable $P\in\P$ or $J_P[r]$ for some indecomposable $P\in\P$ and some $r\in\{0,1,\cdots,m-2\}$; and all indecomposable injective objects in $C^m(\P)$ are of the form $S_P$ for some indecomposable $P\in\P$ or $J_P[r]$ for some indecomposable $P\in\P$ and some $r\in\{0,1,\cdots,m-2\}$. Thus, for $m\geq2$, $C^m(\P)$ is not Frobenius.
%By \cite[Proposition 3.6]{Bau2}, for any object $M_{\bullet}=(M_i,d_i)\in C^m(\P)$ we have the following projective resolution:
%\begin{equation}\xymatrix{0\ar[r]&X_{m-1}\ar[r]&X_{m-2}\ar[r]&\cdots\ar[r]&X_{1}\ar[r]&X_{0}\ar[r]&M_{\bullet}\ar[r]&0}\end{equation}
%with $X_{m-1}=T_{M_0}$ and $X_j=\bigoplus\limits_{i=0}^{m-j-2}J_{M_{m-i-j-2}}[i]\bigoplus T_{M_{m-j-1}}$ for $0\leq j<m-1$.
For $m\geq2$ and any morphism $f:Q\ra P$~of projectives, define $T_f=(M_i,d_i)\in \mathcal{C}^m(\P)$ by
$$M_i=\begin{cases} Q\;\;&\text{~$i=m-1$};\\
                    P\;\;&\text{~$i=m$};\\
                     0 &\text{otherwise,}\end{cases}\qquad
 d_i=\begin{cases} f\;\;&\text{~$i=m-1$};\\
                     0 &\text{otherwise.}\end{cases}$$
So for each projective object $P$, we have that $T_{\id_P}=J_P$.

\section{Hall algebras of $m$-cyclic complexes and $m$-term complexes}
In this section, let $m\geq1$, we study the relation between Hall algebras of $m$-cyclic complexes and $m$-term complexes.

Given objects $L,M,N \in \mathcal{A}$, let $\Ext_\mathcal{A}^1(M,N)_L \subset \Ext_\mathcal{A}^1(M,N)$ be the subset consisting of those equivalence classes of short exact sequences with middle term $L$.
\begin{definition}\label{Hall algebra of abelian category}
The \emph{Hall algebra} $\mathcal {H}(\mathcal{A})$ of $\mathcal{A}$ is the vector space over $\mathbb{C}$ with basis elements $[M] \in \Iso(\mathcal{A})$, and with the multiplication defined by
\[[M] \diamond [N] = \sum\limits_{[L] \in \Iso(\mathcal{A})} {\frac{{|\Ext_\mathcal{A}^1{{(M,N)}_L}|}}{{|\Hom_\mathcal{A}(M,N)|}}} [L].\]
\end{definition}
\begin{remark}
Given objects $L,M,N\in \A$, set
$$g_{MN}^L:=|\{N'\subset L~|~N'\cong N, L/N'\cong M\}|.$$ It is clear that
$$g_{MN}^L=\frac{|W_{MN}^L|}{a_Ma_N},$$ where
$$W_{MN}^L=\{(\varphi,\psi)~|~\xymatrix{0\ar[r]&N\ar[r]^-{\varphi}&L\ar[r]^-{\psi}&M\ar[r]&0}~\text{is~exact~in}~\A\}.$$
By the Riedtmann--Peng formula \cite{Riedtmann,Peng},
\begin{equation}\label{RPGS}g_{MN}^{L}=\frac{|\Ext^1_{\A}(M,N)_{L}|}{|\Hom_{\A}(M,N)|}\cdot \frac{a_{L}}{a_{M}a_{N}}.\end{equation}
Thus, \begin{equation}\label{djdy}[M] \diamond [N] = \sum\limits_{[L] \in \Iso(\mathcal{A})} {\frac{{|W_{MN}^L|}}{{a_L}}} [L].\end{equation}
In fact, in terms of alternative generators $[[M]]=\frac{[M]}{a_M}$, the product takes the form
$$[[M]]\diamond [[N]]= \sum\limits_{[L] \in \Iso(\mathcal{A})}g_{MN}^L[[L]],$$
which is the definition used, for example, in \cite{R90a,Sc}.
\end{remark}

Let $\H(C_m(\A))$ (resp. $\H(C^m(\A))$) be the Hall algebra of the abelian category $C_m(\A)$ (resp. $C^m(\A)$) as defined in Definition \ref{Hall algebra of abelian category}. Let $\H(C_m(\P))$ (resp. $\H(C^m(\P))$) be the subspace of $\H(C_m(\A))$ (resp. $\H(C^m(\A))$) spanned by the isomorphism classes of objects in $C_m(\P)$ (resp. $C^m(\P)$). Since $C_m(\P)$ (resp. $C^m(\P)$) is closed under extensions, $\H(C_m(\P))$ (resp. $\H(C^m(\P))$) is a subalgebra of the Hall algebra $\H(C_m(\A))$ (resp. $\H(C^m(\A))$).

Define $\mathcal {I}$ to be the subspace of $\H(C_m(\P))$ spanned by elements $[M_\circ]$, where $M_\circ=(M_i,d_i)\in C_m(\P)$ with $d_0\neq0$. Denote by $\mathcal {I}'$ the complement space of $\mathcal {I}$ in $\H(C_m(\P))$, which is spanned by elements $[M_\circ]$ satisfying that $M_\circ=(M_i,d_i)\in C_m(\P)$ with $d_0=0$. Namely, as vector spaces, we have the decomposition $\H(C_m(\P))=\mathcal {I}\oplus\mathcal {I}'$.
\begin{lemma}\label{lixiang}
$\mathcal {I}$ is an ideal of $\H(C_m(\P))$.
\end{lemma}
\begin{proof}
We only prove for $m>1$, since the case for $m=1$ can be similarly proved.
Let $m>1$ and $N_\circ=(N_i,c_i)\in C_m(\P)$ with $c_0\neq0$, i.e., $[N_\circ]\in\mathcal {I}$. For any $M_\circ=(M_i,d_i)\in C_m(\P)$, consider the short exact sequence
$$\xymatrix{0\ar[r]&N_\circ\ar[r]^-f&L_\circ\ar[r]^-g&M_\circ\ar[r]&0,}$$
where $L_\circ=(L_i,b_i)\in C_m(\P)$. Then we have that $f_1c_0=b_0f_0$. Assume that $b_0=0$, thus $f_1c_0=0$. We get that $c_0=0$, since $f_1$ is injective. This is a contradiction, and thus we conclude that $b_0\neq0$. That is, $[M_\circ]\diamond[N_\circ]\in\mathcal {I}$. Similarly, we prove that $[N_\circ]\diamond[M_\circ]\in\mathcal {I}$.
Hence, $\mathcal {I}$ is an ideal of $\H(C_m(\P))$.
\end{proof}

Given $M_\circ=(M_i,d_i)\in C_m(\P)$, we fix an object $M_\bullet\in C^m(\P)$ defined by
$$M_\bullet=\xymatrix{M_1\ar[r]^-{d_1}&M_2\ar[r]^-{d_2}&\cdots \ar[r]&M_{m-1}\ar[r]^-{d_{m-1}}&M_{m},}$$
where $M_m:=M_0$. In particular, for $M_\circ=(M_0,d_0)\in C_1(\P)$, $M_\bullet:=M_0\in C^1(\P)$.
\begin{theorem}
$(1)$~There exists a surjective homomorphism of algebras
$$\xymatrix{\chi:\H(C_m(\P))\ar@{->>}[r]&\H(C^m(\P)).}$$

$(2)$~There exists an isomorphism of algebras $$\xymatrix{\rho:\H(C_m(\P))/\mathcal {I}\ar[r]^-{\cong}&\H(C^m(\P)).}$$
\end{theorem}
\begin{proof}
$(1)$~Let $\chi:\H(C_m(\P))\longrightarrow\H(C^m(\P))$ be the linear map defined on basis elements by
$$\chi([M_\circ])=\begin{cases} [M_\bullet]\;\;&\text{if~$d_0=0$}\\
                     0 &\text{otherwise,}\end{cases}$$
where $M_\circ=(M_i,d_i)\in C_m(\P)$.
For any objects $M_\circ=(M_i,d_i), N_\circ=(N_i,c_i)$ in $C_m(\P)$, if $M_\circ\cong N_\circ$, then $d_0=0$ if and only if $c_0=0$; assume that $d_0=0$, in this case $M_\bullet\cong N_\bullet$. Hence, $\chi$ is well defined.

Given objects $M_\circ=(M_i,d_i),N_\circ=(N_i,c_i)\in C_m(\P)$, on the one hand, if $d_0\neq0$ or $c_0\neq0$, i.e., $\chi([M_\circ])=0$ or $\chi([N_\circ])=0$, thus $\chi([M_\circ])\diamond\chi([M_\circ])=0$. By Lemma \ref{lixiang} we have that $[M_\circ]\diamond[N_\circ]\in\mathcal {I}$, thus $\chi([M_\circ]\diamond[N_\circ])=0$.

On the other hand, if $d_0=0$ and $c_0=0$, then by (\ref{djdy}), we obtain that
$$[M_\circ]\diamond[N_\circ]=\sum\limits_{[L_\circ]:L_\circ=(L_i,b_i)~~\text{with}~~b_0=0}\frac{|W_{M_\circ N_\circ}^{L_\circ}|}{a_{L_\circ}}[L_\circ]+x,$$ where $x\in\mathcal {I}$. Since $\chi(x)=0$, we obtain that $$\chi([M_\circ]\diamond[N_\circ])=\sum\limits_{[L_\circ]:L_\circ=(L_i,b_i)~~\text{with}~~b_0=0}\frac{|W_{M_\circ N_\circ}^{L_\circ}|}{a_{L_\circ}}[L_\bullet]$$
It is clear that there exists a bijection between $W_{M_\circ N_\circ}^{L_\circ}$ and $W_{M_\bullet N_\bullet}^{L_\bullet}$, and $a_{L_\circ}=a_{L_\bullet}$. Hence, 
\begin{flalign*}\sum\limits_{[L_\circ]:L_\circ=(L_i,b_i)~~\text{with}~~b_0=0}\frac{|W_{M_\circ N_\circ}^{L_\circ}|}{a_{L_\circ}}[L_\bullet]&=\sum\limits_{[L_\bullet]}\frac{|W_{M_\bullet N_\bullet}^{L_\bullet}|}{a_{L_\bullet}}[L_\bullet]\\&=\chi([M_\circ])\diamond\chi([M_\circ]).\end{flalign*}
Therefore, $\chi$ is a homomorphism of algebras.

For any $M_\bullet=(M_i,d_i)\in C^m(\P)$, take $M_\circ=(M_i,d_i)\in C_m(\P)$ with $M_0:=M_m$ and $d_0=0$, then $\chi([M_\circ])=[M_\bullet]$, and thus $\chi$ is surjective.

$(2)$~Since $\chi(\mathcal {I})=0$, $\chi$ induces a surjective homomorphism of algebras
$$\xymatrix{\rho:\H(C_m(\P))/\mathcal {I}\ar@{->>}[r]&\H(C^m(\P)).}$$
The injectivity of $\rho$
follows from the fact that $\rho$ sends the basis $\{[M_\circ]+\mathcal {I}~|~M_\circ=(M_i,d_i)\in C_m(\P)~~\text{with}~~d_0=0\}$ of $\H(C_m(\P))/\mathcal {I}$ to a basis of $\H(C^m(\P))$.
\end{proof}
%\begin{theorem}
%There exists an isomorphism of algebras
%$$\xymatrix{\chi:\H(C_m(\P))/\mathcal {I}\ar[r]&\H(C^m(\P)).}$$
%\end{theorem}
%\begin{proof}
%For any $M_\circ=(M_i,d_i)\in C_m(\P)$, if $d_0\neq0$, define $\chi([M_\circ]+\mathcal {I})=0$;
%otherwise, define $\chi([M_\circ]+\mathcal {I})=[M_\bullet]$.
%Thus we obtain a linear map $\chi$ by linearly extending rule.
%
%For any $M_\circ=(M_i,d_i),N_\circ=(N_i,c_i),L_\circ=(L_i,b_i)\in C_m(\P)$ with $d_0,c_0,b_0=0$, it is clear that there exists a bijection between $W_{M_\circ N_\circ}^{L_\circ}$ and $W_{M_\bullet N_\bullet}^{L_\bullet}$, and $a_{L_\circ}=a_{L_\bullet}$. By (\ref{djdy}), we obtain that $\chi$ is a homomorphism of algebras.
%
%For any $M_\bullet=(M_i,d_i)\in C^m(\P)$, take $M_\circ=(M_i,d_i)\in C_m(\P)$ with $M_0:=M_m$ and $d_0=0$, then $\chi([M_\circ])=[M_\bullet]$, and thus $\chi$ is surjective. The injectivity of $\chi$
%follows from the fact that $\chi$ sends the basis $\{[M_\circ]+\mathcal {I}~|~M_\circ=(M_i,d_i)\in C_m(\P)~~\text{with}~~d_0=0\}$ of $\H(C_m(\P))/\mathcal {I}$ to the basis $\{[M_\bullet]\}$ of $\H(C^m(\P))$.
%\end{proof}
%Let $\xymatrix{\pi:\H(C_m(\P))\ar@{->>}[r]&\H(C_m(\P))/\mathcal {I}}$ be the natural quotient map. Then we have the following
%\begin{corollary}
%There exists a surjective homomorphism of algebras
%$$\xymatrix{\rho:\H(C_m(\P))\ar@{->>}[r]&\H(C^m(\P)).}$$
%\end{corollary}
%\begin{proof}
%Taking $\rho=\chi\circ\pi$ gives the desired homomorphism.
%\end{proof}

\section{Indecomposable objects in $C_m(\P)$ and $C^m(\P)$}
From now onwards, we always assume that $\A$ is hereditary until the end of the entire paper. For each object $M\in\A$, according to \cite[Section 4.1]{Br}, it has a minimal projective resolution\footnote{The notations $P_M$ and $\Omega_M$ will be used throughout the paper.}
\begin{equation}\label{mpr}
\xymatrix{0\ar[r]&\Omega_M\ar[r]^{\delta_M}&P_M\ar[r]&M\ar[r]&0.}
\end{equation}
Moreover, we have the following well-known
\begin{lemma}\label{jx}
Given $M\in\A$, each projective resolution of $M$ is isomorphic to a resolution of the form
\begin{equation*}
\xymatrix{0\ar[r]&\Omega_M\oplus R\ar[r]^{\delta_M\oplus1}&P_M\oplus R\ar[r]&M\ar[r]&0,}
\end{equation*}
for some $R\in\P$ and some minimal projective resolution \begin{equation*}
\xymatrix{0\ar[r]&\Omega_M\ar[r]^{\delta_M}&P_M\ar[r]&M\ar[r]&0.}
\end{equation*}
\end{lemma}

We define objects $C_M:=C_{\delta_M}\in C_m(\P)$ for $m\geq0$ and $T_M:=T_{\delta_M}\in C^m(\P)$ for $m\geq2$. By Lemma \ref{jx}, we know that any two minimal projective resolutions of $M$ are isomorphic, so $C_M$ and $T_M$ are well defined up to isomorphism.

Now, let us give characterizations of indecomposable objects in $C_m(\P)$ and $C^m(\P)$.
\begin{lemma}{\rm(\cite[Lemma 2.3]{ChenD})}\label{dx1}
For $m\geq0$,
the objects $C_M[r]$ and $K_P[r]$, where $r\in\mathbb{Z}_m$, $M\in\A$ is indecomposable and $P\in\P$ is indecomposable, provide a complete set of indecomposable objects in $C_m(\P)$. Moreover, all $K_P[r]$ are the whole indecomposable projective-injective objects in $C_m(\P)$.
\end{lemma}
\begin{lemma}\label{indobj}
For $m\geq2$,
the objects $S_P$, $T_M[r]$ and $J_P[r]$, where $0\leq r<m-1$, $M\in\A$ is indecomposable and $P\in\P$ is indecomposable, provide a complete set of indecomposable objects in $C^m(\P)$. Moreover, all $J_P[r]$ and $T_P$ are the whole indecomposable projective objects; all $J_P[r]$ and $S_P$ are the whole indecomposable injective objects.
\end{lemma}
\begin{proof}
We only prove the first statement, since the others have been given in the previous section.
Take an arbitrary object in $C^m(\P)$
$$M_{\bullet}=\xymatrix{M_1\ar[r]^-{d_1}&M_2\ar[r]^-{d_2}&\cdots \ar[r]&M_{m-1}\ar[r]^-{d_{m-1}}&M_{m}.}$$
For each $1\leq i<m$, we have the short exact sequence
$$\xymatrix{0\ar[r]&\Ker d_i\ar[r]^-{\lambda_i}&M_i\ar[r]^-{d_i}&\im d_i\ar[r]&0.}$$
By the hereditary assumption, all the objects appearing in these sequences are projective. Thus the sequences split. That is,
we can assume that $M_i=\Ker d_i\oplus\im d_i$ for each $1\leq i<m$. Writing $M_{\bullet}$ as follows:
$$M_{\bullet}=\xymatrix{\Ker d_1\oplus\im d_1\ar[r]^-{\left({\begin{smallmatrix}0&0\\0&d'_1\end{smallmatrix}}\right)}&\im d_2\oplus\Ker d_2\ar[r]^-{(d'_2,0)}&M_3\ar[r]^-{d_3}&\cdots \ar[r]&M_{m-1}\ar[r]^-{d_{m-1}}&M_{m},}$$
where and elsewhere $d'_i=d_i|_{\im d_i}$ for $1\leq i<m$, we obtain that
$M_\bullet$ is the direct sum of the following two objects
$$M^1_{\bullet}=\xymatrix{\Ker d_1\ar[r]^-{0}&\im d_2\ar[r]^-{d'_2}&M_3\ar[r]^-{d_3}&\cdots \ar[r]&M_{m-1}\ar[r]^-{d_{m-1}}&M_{m}}$$
and $$N^1_{\bullet}=\xymatrix{\im d_1\ar[r]^-{d'_1}&\Ker d_2\ar[r]&0\ar[r]&\cdots \ar[r]&0\ar[r]&0.}$$
Similarly, $M^1_{\bullet}$ is the direct sum of the following two objects
$$M^2_{\bullet}=\xymatrix{\Ker d_1\ar[r]&0\ar[r]&\im d_3\ar[r]^-{d'_3}&M_4\ar[r]^-{d_4}&\cdots \ar[r]&M_{m-1}\ar[r]^-{d_{m-1}}&M_{m}}$$
$$N^2_{\bullet}=\xymatrix{0\ar[r]&\im d_2\ar[r]^-{d'_2}&\Ker d_3\ar[r]&0\ar[r]&\cdots \ar[r]&0.}$$
Repeating this process, we get that $M_\bullet$ has a direct sum decomposition
$$M_\bullet=S_{P}\oplus N^1_{\bullet}\oplus\cdots\oplus N^{m-2}_{\bullet}\oplus N^{m-1}_{\bullet}$$
where $P=\Ker d_1$; for $1\leq i\leq m-2$, $N^i_{\bullet}=(X_j,f_j)$ with $X_j=0$ if $j\notin\{i,i+1\}$, $X_{i}=\im d_{i}$, $X_{i+1}=\Ker d_{i+1}$ and $f_{i}=d'_{i}$; and $N^{m-1}_{\bullet}=(X_j,f_j)$ with $X_j=0$ if $j\notin\{m-1,m\}$, $X_{m-1}=\im d_{m-1}$, $X_{m}=M_{m}$ and $f_{m-1}=d'_{m-1}$.

Noting for each $1\leq i\leq m-1$ the differential $f_{i}$ in $N^i_{\bullet}$ is injective, we have the short exact sequence
$$\xymatrix{0\ar[r]&X_{i}\ar[r]^-{f_{i}}&X_{i+1}\ar[r]&Y_i\ar[r]&0,}$$
where $Y_i:=\Coker f_{i}$. Then, by Lemma \ref{jx}, $N^i_{\bullet}\cong T_{Y_i}[m-i-1]\oplus J_{R_i}[m-i-1]$ for some $R_i\in\P$. Set $r=m-i-1$, since $1\leq i\leq m-1$ we have that $0\leq r<m-1$. Therefore, we complete the proof.\end{proof}

\section{Hall algebras of bounded complexes and derived Hall algebras}
In order to study the Hall algebra of $C^m(\P)$, by reformulating \cite[Theorem 3.2]{ZHC2} we first give a characterization of the Hall algebra of $C^b(\P)$, and then relate it to the derived Hall algebra of $\A$. These are similar to the results given in \cite[Section 5.2]{LinP}.

Let $\H(C^b(\A))$ be the Hall algebra of the abelian category $C^b(\A)$ as defined in Definition \ref{Hall algebra of abelian category}. Let $\H(C^b(\P))$ be the subspace of $\H(C^b(\A))$ spanned by the isomorphism classes of objects in $C^b(\P)$. Since $C^b(\P)$ is closed under extensions, $\H(C^b(\P))$ is a subalgebra of the Hall algebra $\H(C^b(\A))$.

For objects $M,N \in \mathcal{A}$, define \begin{equation}\label{Euler form}
\lr{M,N}:=\dim_k\Hom_{\A}(M,N)-\dim_k\Ext^1_{\A}(M,N).\end{equation}
Since $\A$ is hereditary,
it descends to give a bilinear form
$$\lr{\cdot ,\cdot }: K(\mathcal{A})\times K(\mathcal{A})\longrightarrow \mathbb{Z},$$ known as the \emph{Euler form}.

Define the Hall algebra $\H_{\tw}(\A)$ to be the same vector space as $\H(\A)$, but with the twisted multiplication defined by $$[M]*[N]=q^{\lr{M,N}}\cdot[M]\diamond[N].$$

Given objects $M_\bullet,N_\bullet\in C^b(\P)$, there exists a positive integer $m$ such that $M_\bullet,N_\bullet\in C^m(\P)$. Since the global dimension of $C^m(\P)$ is $m-1$, we obtain that there exists a positive integer $N$ such that for all $i>N$, $\Ext^i_{C^b(\P)}(M_\bullet,N_\bullet)=0$. That is, $C^b(\P)$ is locally homological finite. Thus, the Euler form of $C^b(\P)$
\begin{equation*}
\lr{\cdot ,\cdot }: K(C^b(\P))\times K(C^b(\P))\longrightarrow \mathbb{Z}
\end{equation*}
determined by \begin{equation}\lr{M_\bullet,N_\bullet}=\sum\limits_{i\geq0}(-1)^i\dim_k\Ext^i_{C^b(\P)}(M_\bullet,N_\bullet)\end{equation}
is also well defined.
\begin{lemma}\label{ola}
For any $M,N\in\A$ and $r,l\in\mathbb{Z}$, we have that
\begin{itemize}
\item[(1)] $\Ext_{C^b(\P)}^i(C_M[r],C_N[r])=0$ for any $i\geq2$;
\item[(2)] $\Ext_{C^b(\P)}^1(C_M[r],C_N[r])\cong\Ext_{\A}^1(M,N)$;
\item[(3)] $\Ext_{C^b(\P)}^i(C_M[r],C_N[r+1])=0$ for any $i\geq1$;
\item[(4)] $\lr{C_M[r],C_N[l]}=(-1)^{r-l}\lr{M,N}$, if $r-l\geq1$;
\item[(5)] $\lr{C_M[r],C_N[l]}=0$, if $l-r>1$.
\end{itemize}
\end{lemma}
\begin{proof}
For any $i\geq1$, by \cite[Lemma 3.1]{Gor2}, we have that
\begin{equation}\label{ext}
\begin{split}
\Ext_{C^b(\P)}^i(C_M[r],C_N[l])&\cong\Hom_{K^b(\P)}(C_M[r],C_N[l+i])\\
&\cong\Hom_{D^b(\A)}(M,N[l-r+i])\\
&\cong\begin{cases} \Hom_{\A}(M,N)\;\;&\text{~$l-r+i=0$};\\
                    \Ext_{\A}^1(M,N)\;\;&\text{~$l-r+i=1$};\\
                     0 &\text{otherwise.}\end{cases}
\end{split}\end{equation}
Thus, $(1$-$3)$ can be easily proved. Noting that $\Hom_{C^b(\P)}(C_M[r],C_N[l])=0$ if $r-l\geq1$, we prove $(4)$;
Noting that $\Hom_{C^b(\P)}(C_M[r],C_N[l])=0$ if $l-r>1$, we prove $(5)$.
\end{proof}

Define the Hall algebra $\H_{\tw}(C^b(\P))$ to be the same vector space as $\H(C^b(\P))$, but with the twisted multiplication
\begin{equation}[M_\bullet]\ast[N_\bullet]=q^{\lr{M_\bullet,N_\bullet}}\cdot[M_\bullet]\diamond[N_\bullet].\end{equation}
For any $P\in\P$ and $r\in\mathbb{Z}$, since $K_P[r]$ is projective-injective in $C^b(\P)$, we easily obtain that \begin{equation}[K_P[r]]\ast[M_\bullet]=[M_\bullet\oplus K_P[r]]=[M_\bullet]\ast[K_P[r]]\end{equation} for all $M_\bullet\in C^b(\P)$.

%A complex $K_\bullet\in C^b(\A)$ is called \emph{acyclic} if $H_\ast(K_\bullet)=0$. Given $K_\bullet\in C^b(\P)$, by \cite[Lemma 2.3]{ZHC2} we know that $K_\bullet$ is acyclic if and only if there are finitely many objects $Q_r\in\P$, unique up to isomorphism, such that $K_\bullet\cong \bigoplus\limits_{r\in\mathbb{Z}}K_{Q_r}[r]$. For any acyclic complex $K_\bullet\in C^b(\P)$, we have that \begin{equation}[K_\bullet]\ast[M_\bullet]=[M_\bullet\oplus K_\bullet]=[M_\bullet]\ast[K_\bullet]\end{equation} for all $M_\bullet\in C^b(\P)$.

Define the Hall algebra $\M\H(\A)$ to be the localization of $\H_{\tw}(C^b(\P))$ with respect to elements $[K_P[r]]$ for all $P\in\P$ and $r\in\mathbb{Z}$.
For each $r\in\mathbb{Z}$ and $\alpha\in K(\A)$, by writing $\alpha=\hat{P}-\hat{Q}$ for some $P,Q\in\P$,
we define $$K_{\alpha,r}=[K_P[r]]\ast[K_Q[r]]^{-1}.$$
For any $\alpha,\beta\in K(\A)$, it is easy to see that
\begin{equation}\label{kejia}
K_{\alpha,r}\ast K_{\beta,r}=K_{\alpha+\beta,r}.\end{equation}
Moreover, for any $M_\bullet\in C^b(\P)$ we have that
\begin{equation}
K_{\alpha,r}\ast[M_\bullet]=[M_\bullet]\ast K_{\alpha,r}.
\end{equation}

Given an object $M\in\A$, for each $r\in\mathbb{Z}$, we define
\begin{equation}
E_{M,r}:=K_{-\hat{\Omega}_M,r}\ast[C_M[r]]\in\M\H(\A).
\end{equation}
%By \cite{Br}, $E_{M,r}$ does not depend on the minimality of projective resolutions of $M$.
Let us reformulate \cite[Proposition 4.4(2)]{ChenD} in the following
\begin{proposition}\label{main1}
For each $r\in\mathbb{Z}$, there exists an embedding of algebras
$$\xymatrix{\psi_r:\H_{\tw}(\A)\ar@{^{(}->}[r]&\M\H(\A),} \xymatrix{[M]\ar@{|->}[r]&E_{M,r}.}$$
\end{proposition}
\begin{remark}
Compared with the modified Ringel--Hall algebra studied in \cite{LinP}, the element $E_{M,r}$ corresponding to $[M]$ needs to be defined by appending the element $K_{-\hat{\Omega}_M,r}$ to $[C_M[r]]$. Actually, in the modified Ringel--Hall algebra, the embedding above is immediate.
\end{remark}
By Lemma \ref{dx1}, applying the arguments similar to those in the proof of \cite[Proposition 4.4(3)]{ChenD}, we obtain
the following
\begin{proposition}
$\M\H(\A)$ has a basis consisting of elements
$$K_{\alpha_r,r}\ast K_{\alpha_{r+1},r+1}\ast\cdots\ast K_{\alpha_l,l}\ast E_{M_r,r}\ast E_{M_{r+1},r+1}\ast\cdots\ast E_{M_l,l},$$
where $r,l\in\mathbb{Z}$, $r\leq l$, $\alpha_i\in K(\A)$ and $M_i\in\A$ for $r\leq i\leq l$.
\end{proposition}
Given objects $M,N,X,Y\in\A$, we denote by $W_{MN}^{XY}$ the set $$\{(f,g,h)~|~\xymatrix{0\ar[r]&X\ar[r]^g&M\ar[r]^f&N\ar[r]^h&Y\ar[r]&0} \text{is exact in}~~\A\}$$ and set $$\gamma_{MN}^{XY}:=\frac{|W_{MN}^{XY}|}{a_Ma_N}.$$

Applying the arguments similar to those in the proof of \cite[Theorem 3.2]{ZHC2} together with Lemma \ref{ola}, we obtain
the following
\begin{proposition}\label{guanxi}
The Hall algebra $\M\H(\A)$ is generated by the elements in
$$\{E_{M,r},K_{\alpha,r}~|~M\in\A,\alpha\in K(\A),r\in\mathbb{Z}\}$$
with the defining relations
\begin{flalign}
&E_{M,r}\ast E_{N,r}=\sum\limits_{[L]}q^{\lr{M,N}}{\frac{{|\Ext_\mathcal{A}^1{{(M,N)}_L}|}}{{|\Hom_\mathcal{A}(M,N)|}}}E_{L,r};\\
&E_{M,r+1}\ast E_{N,r}=\sum\limits_{[X],[Y]}q^{-\lr{M,N}}\gamma_{MN}^{XY}\frac{a_Ma_N}{a_Xa_Y}E_{Y,r}\ast E_{X,r+1}\ast K_{\hat{M}-\hat{X},r};\\
&E_{M,r}\ast E_{N,l}=q^{(-1)^{r-l}\lr{M,N}}E_{N,l}\ast E_{M,r},~~~~r-l>1;\\
&K_{\alpha,r}\ast E_{M,l}=E_{M,l}\ast K_{\alpha,r};\\
&K_{\alpha,r}\ast K_{\beta,r}=K_{\alpha+\beta,r},~~K_{\alpha,r}\ast K_{\beta,l}=K_{\beta,l}\ast K_{\alpha,r},
\end{flalign}
where $M,N\in\A$, $\alpha,\beta\in K(\A)$ and $r,l\in\mathbb{Z}$.
\end{proposition}
\begin{remark}
By Proposition \ref{guanxi} and \cite[Proposition 5.3]{LinP}, we obtain that the Hall algebra $\M\H(\A)$ is isomorphic to the modified Ringel--Hall algebra $\M\H_{\tw}(\A)$ defined in \cite{LinP}. Explicitly, there exists an isomorphism $\varphi:\M\H(\A)\longrightarrow\M\H_{\tw}(\A)$ defined on generators by
$E_{M,r}\mapsto U_{M,-r}$ and $K_{\alpha,r}\mapsto K_{\alpha,-r}$.
\end{remark}
For simplicity, we recall the twisted derived Hall algebra $\mathcal {D}\mathcal {H}(\A)$ of $\A$ in the form of generators and relations in the following
\begin{proposition}{\rm(\cite{Toen2006})}
$\mathcal {D}\mathcal {H}(\A)$ is an associative and unital $\mathbb{C}$-algebra generated by the elements in $\{Z_M^{[r]}~|~M\in\Iso(\A),~r\in \mathbb{Z}\}$ and the following relations
\begin{flalign}
&Z_M^{[r]}\ast Z_N^{[r]}=\sum\limits_{[L]}q^{\lr{M,N}}{\frac{{|\Ext_\mathcal{A}^1{{(M,N)}_L}|}}{{|\Hom_\mathcal{A}(M,N)|}}}Z_L^{[r]};\\
&Z_M^{[r+1]}\ast Z_N^{[r]}=\sum\limits_{[X],[Y]}q^{-\lr{M,N}}\gamma_{MN}^{XY}\frac{a_Ma_N}{a_Xa_Y} Z_Y^{[r]}\ast Z_X^{[r+1]};\\
&Z_M^{[r]}\ast Z_N^{[l]}=q^{(-1)^{r-l}\lr{M,N}} Z_N^{[l]}\ast Z_M^{[r]}, \quad r-l>1.
\end{flalign}
\end{proposition}
Now we reformulate \cite[Theorem 5.5]{LinP} in the following
\begin{theorem}
There is an embedding of algebras
$$\xymatrix{\Psi:\mathcal {D}\mathcal {H}(\A)\ar@{^{(}->}[r]&\M\H(\A)}$$
defined on generators (with $n>0$) by
\begin{flalign*}
Z_M^{[0]}\mapsto E_{M,0},~~~~Z_M^{[n]}\mapsto E_{M,n}\ast\prod\limits_{i=1}^nK_{(-1)^i\hat{M},n-i}~~~~\text{and}~~~~Z_M^{[-n]}\mapsto E_{M,-n}\ast\prod\limits_{i=1}^nK_{(-1)^i\hat{M},-(n-i+1)}.
\end{flalign*}
\end{theorem}
Let $\mathbb{T}^\infty(\A)$ be the subalgebra of $\M\H(\A)$ generated by elements $K_{\alpha,r}$ with $\alpha\in K(\A)$ and $r\in\mathbb{Z}$.
By \cite[Corollary 5.6]{LinP},
there is an isomorphism of algebras
\begin{equation}
\hat{\Psi}:\mathcal {D}\mathcal {H}(\A)\otimes_{\mathbb{C}}\mathbb{T}^\infty(\A)\longrightarrow\M\H(\A), x\otimes t\mapsto
\Psi(x)\ast t.
\end{equation}
By \cite[Corollary 5.7]{LinP}, we know that $\M\H(\A)$ is invariant under derived equivalences. The inverse of $\hat{\Psi}$
is the homomorphism \begin{equation}\label{tg}
\hat{\Psi}^{-1}:\M\H(\A)\longrightarrow\mathcal {D}\mathcal {H}(\A)\otimes_{\mathbb{C}}\mathbb{T}^\infty(\A)\end{equation} defined on generators (with $n>0$) by
$$K_{\alpha,n}\mapsto K_{\alpha,n},\quad E_{M,0}\mapsto Z_M^{[0]},$$
$$E_{M,n}\mapsto Z_M^{[n]}\ast\prod\limits_{i=0}^{n-1}K_{(-1)^{n-i-1}\hat{M},i},~~\text{and}~~E_{M,-n}\mapsto Z_M^{[-n]}\ast\prod\limits_{i=1}^{n}K_{(-1)^{n-i}\hat{M},-i},$$
where we have written elements $x\otimes t\in\mathcal {D}\mathcal {H}(\A)\otimes_{\mathbb{C}}\mathbb{T}^\infty(\A)$ as $x\ast t$.

In fact, these results above are essentially the same as those given by Gorsky in \cite[Theorem 4.2]{Gor2}. However, the explicit map between the Hall algebra of $C^b(\P)$ and the derived Hall algebra is not given there, besides, the twist therein used only in the Hall algebra of $C^b(\P)$, not in the derived Hall algebra, not only involves $C^b(\P)$, but also $K^b(\P)$.

\section{Hall algebras of $m$-term complexes}
Since the global dimension of $C^m(\P)$ is $m-1$,  the Euler form of $C^m(\P)$
\begin{equation*}
\lr{\cdot ,\cdot }: K(C^m(\P))\times K(C^m(\P))\longrightarrow \mathbb{Z}
\end{equation*}
determined by \begin{equation}\lr{M_\bullet,N_\bullet}=\sum\limits_{i=0}^{m-1}(-1)^i\dim_k\Ext^i_{C^m(\P)}(M_\bullet,N_\bullet)\end{equation}
is well defined.

Define the Hall algebra $\H_{\tw}(C^m(\P))$ to be the same vector space as $\H(C^m(\P))$, but with the twisted multiplication
\begin{equation}[M_\bullet]\ast[N_\bullet]=q^{\lr{M_\bullet,N_\bullet}}\cdot[M_\bullet]\diamond[N_\bullet].\end{equation}
For any projective-injective object $J_\bullet\in C^m(\P)$, we easily obtain that \begin{equation}[J_\bullet]\ast[M_\bullet]=[M_\bullet\oplus J_\bullet]=[M_\bullet]\ast[J_\bullet]\end{equation} for all $M_\bullet\in C^m(\P)$. Define the Hall algebra $\M\H_m(\A)$ to be the localization of $\H_{\tw}(C^m(\P))$ with respect to elements $[J_\bullet]$ corresponding to projective-injective objects $J_\bullet$ in $C^m(\P)$.
It is easy to see that $\M\H_1(\A)$ is isomorphic to the group algebra $\mathbb{C}[K(\A)]$ of the Grothendieck group $K(\A)$.

From now on, we let $m\geq2$. As before, for each $0\leq r<m-1$ and $\alpha\in K(\A)$, by writing $\alpha=\hat{P}-\hat{Q}$ for some $P,Q\in\P$,
we define $$J_{\alpha,r}=[J_P[r]]\ast[J_Q[r]]^{-1}.$$
For any $\alpha,\beta\in K(\A)$, it is easy to see that
\begin{equation}\label{kejia}
J_{\alpha,r}\ast J_{\beta,r}=J_{\alpha+\beta,r}.\end{equation}
Moreover, for any $M_\bullet\in C^m(\P)$ we have that
\begin{equation}
J_{\alpha,r}\ast[M_\bullet]=[M_\bullet]\ast J_{\alpha,r}.
\end{equation}

Given an object $M\in\A$, for each $0\leq r<m-1$, we define
\begin{equation}
\mathbb{X}_{M,r}:=J_{-\hat{\Omega}_M,r}\ast[T_M[r]]\in\M\H_m(\A).
\end{equation}
For each object $P\in\P$, we define
\begin{equation}\mathbb{X}_{P,m-1}:=[S_P]\in\M\H_m(\A).\end{equation}

First of all, we give a basis in $\M\H_m(\A)$ as follows:
\begin{proposition}\label{mji}
$\M\H_m(\A)$ has a basis consisting of elements
$$J_{\alpha_0,0}\ast J_{\alpha_{1},1}\ast\cdots\ast J_{\alpha_{m-2},m-2}\ast \mathbb{X}_{M_0,0}\ast \mathbb{X}_{M_{1},1}\ast\cdots\ast \mathbb{X}_{M_{m-2},m-2}\ast\mathbb{X}_{P,m-1},$$
where $\alpha_r\in K(\A)$ and $M_r\in\A$ for $0\leq r<m-1$, and $P\in\P$.
\end{proposition}
\begin{proof}
It is similar to (\ref{ext}) that for any objects $M,N\in\A$,
$$\Ext_{C^m(\P)}^1(T_M[r],T_N[l])=0,~~0\leq r<l<m-1.$$
Thus, for $M_0,M_1,\cdots,M_{m-2}\in\A$ and $P\in\P$, noting that $S_P$ is injective in $C^m(\P)$, we have that
$$[T_{M_0}]\ast [T_{M_1}[1]]\ast\cdots\ast [T_{M_{m-2}}[m-2]]\ast [S_P]=q^{a}[T_{M_0}\oplus T_{M_1}[1]\oplus\cdots\oplus T_{M_{m-2}}[m-2]\oplus S_P]$$ for some $a\in\mathbb{Z}$.
By Lemma \ref{indobj}, we can easily complete the proof.
\end{proof}

Let us consider the shift functor $[m]:C^m(\P)\longrightarrow C^b(\P)$. Clearly, $[m]$ is a fully faithful exact functor.
Moreover, for any objects $M_\bullet,N_\bullet\in C^m(\P)$, we have that $\Ext_{C^m(\P)}^i(M_\bullet,N_\bullet)\cong\Ext_{C^b(\P)}^i(M_\bullet[m],N_\bullet[m])$ for all $i\geq1$. That is, $[m]$ is an extremely faithful exact functor. By functorial properties of Hall algebras (cf. \cite{Sc}), there exists an embedding of algebras $\lambda': \H_{\tw}(C^m(\P))\longrightarrow\H_{\tw}(C^b(\P))$.
Thus, we have the following
\begin{proposition}\label{erqian}
There exists an embedding of algebras $\xymatrix{\lambda: \M\H_m(\A)\ar@{^{(}->}[r]&\M\H(\A).}$
\end{proposition}
\begin{proof}
Let $\pi_m:\H_{\tw}(C^m(\P))\longrightarrow\M\H_m(\A)$ and $\pi:\H_{\tw}(C^b(\P))\longrightarrow\M\H(\A)$ be the natural homomorphisms of algebras. Since $\pi\circ\lambda'$ maps all elements $[J_P[r]]$ in $\H_{\tw}(C^m(\P))$ to the invertible elements $[K_P[r]]$ in $\M\H(\A)$, by universal properties of localizations, we obtain that there is a unique homomorphism of algebras
$$\lambda:\M\H_m(\A)\longrightarrow\M\H(\A)$$ such that $\pi\circ\lambda'=\lambda\circ\pi_m$. Explicitly,
$\lambda(J_{\alpha,r})=K_{\alpha,r}$, $\lambda(\mathbb{X}_{M,r})=E_{M,r}$ and $\lambda(\mathbb{X}_{P,m-1})=E_{P,m-1}$
for all $\alpha\in K(\A)$, $M\in\A$, $0\leq r<m-1$ and $P\in\P$. Thus, the injectivity of $\lambda$ follows from the fact that $\lambda$ sends the basis of $\M\H_m(\A)$ in Proposition \ref{mji} to a linearly independent set in $\M\H(\A)$.
\end{proof}
Combining Proposition \ref{erqian} with Proposition \ref{main1}, we obtain the following
\begin{proposition}\label{main2}
For each $0\leq r<m-1$, there exists an embedding of algebras
$$\xymatrix{\varphi_r:\H_{\tw}(\A)\ar@{^{(}->}[r]&\M\H_m(\A),} \xymatrix{[M]\ar@{|->}[r]&\mathbb{X}_{M,r}.}$$
\end{proposition}
\begin{proof}
Taking $\varphi_r=\lambda^{-1}\circ\psi_r$ gives the desired homomorphism.
\end{proof}
Combining Propositions \ref{guanxi}, \ref{mji} and \ref{erqian}, we have the following
\begin{proposition}\label{main3}
The Hall algebra $\M\H_m(\A)$ is generated by the elements in
$$\{J_{\alpha,r},\mathbb{X}_{M,r},\mathbb{X}_{P,m-1}~|~\alpha\in K(\A),M\in\A,P\in\P,0\leq r<m-1\}$$
with the defining relations
\begin{flalign}
&\mathbb{X}_{M,r}\ast \mathbb{X}_{N,r}=\sum\limits_{[L]}q^{\lr{M,N}}{\frac{{|\Ext_\mathcal{A}^1{{(M,N)}_L}|}}{{|\Hom_\mathcal{A}(M,N)|}}}\mathbb{X}_{L,r},
\mathbb{X}_{P,m-1}\ast\mathbb{X}_{Q,m-1}=\mathbb{X}_{P\oplus Q,m-1};\\
&\mathbb{X}_{M,r+1}\ast \mathbb{X}_{N,r}=\sum\limits_{[X],[Y]}q^{-\lr{M,N}}\gamma_{MN}^{XY}\frac{a_Ma_N}{a_Xa_Y}\mathbb{X}_{Y,r}\ast \mathbb{X}_{X,r+1}\ast J_{\hat{M}-\hat{X},r},0\leq r<m-2;\\
&\mathbb{X}_{P,m-1}\ast \mathbb{X}_{M,m-2}=\sum\limits_{[R],[B]}q^{-\lr{P,M}}\gamma_{PM}^{RB}\frac{a_Pa_M}{a_Ra_B}\mathbb{X}_{B,m-2}\ast\mathbb{X}_{R,m-1}\ast J_{\hat{P}-\hat{R},m-2};\\
&\mathbb{X}_{M,r}\ast \mathbb{X}_{N,l}=q^{(-1)^{r-l}\lr{M,N}}\mathbb{X}_{N,l}\ast \mathbb{X}_{M,r},~~~~r-l>1;\\
&\mathbb{X}_{P,m-1}\ast \mathbb{X}_{M,r}=q^{(-1)^{m-r-1}\lr{P,M}}\mathbb{X}_{M,r}\ast \mathbb{X}_{P,m-1},~~~~0\leq r<m-2;\\
&J_{\alpha,r}\ast \mathbb{X}_{M,l}=\mathbb{X}_{M,l}\ast J_{\alpha,r},~~~~J_{\alpha,r}\ast \mathbb{X}_{P,m-1}=\mathbb{X}_{P,m-1}\ast J_{\alpha,r};\\
&J_{\alpha,r}\ast J_{\beta,r}=J_{\alpha+\beta,r},~~J_{\alpha,r}\ast J_{\beta,l}=J_{\beta,l}\ast J_{\alpha,r},
\end{flalign}
where $M,N\in\A$, $P,Q\in\P$ and $\alpha,\beta\in K(\A)$.
\end{proposition}
Let $\mathbb{T}^{m-1}(\A)$ be the subalgebra of $\mathbb{T}^{\infty}(\A)$ generated by elements $K_{\alpha,r}$
with $\alpha\in K(\A)$ and $0\leq r<m-1$.
\begin{corollary}
There is an embedding of algebras
$$\xymatrix{\iota:\M\H_m(\A)\ar@{^{(}->}[r]&\mathcal {D}\mathcal {H}(\A)\otimes_{\mathbb{C}}\mathbb{T}^{m-1}(\A).}$$
\end{corollary}
\begin{proof}
Taking $\iota=\hat{\Psi}^{-1}\circ\lambda$ gives the desired homomorphism.
\end{proof}
\begin{remark}
$\M\H_2(\A)$ is used in \cite{DXZ} to realize quantum cluster algebra with principal coefficients as its subquotient.
\end{remark}

\section{Integration maps associated to Hall algebras of $C^2(\P)$}
In this section, let $\mathcal {E}$ be a finitary exact category of global dimension at most one. By \cite{Hubery},  Definition \ref{Hall algebra of abelian category} also applies to the exact category $\mathcal {E}$. We also have the  Riedtmann--Peng formula $(\ref{RPGS})$ in the Hall algebra $\H(\mathcal {E})$ of $\mathcal {E}$.

Let $\mathcal{T}_q(\mathcal {E})$ be the $\mathbb{Z}[q^{\pm}]$-algebra with a basis $\{X^{\alpha}~|~\alpha\in K(\mathcal {E})\}$ and the
multiplication given by
\[X^{\alpha}\ast X^{\beta}=q^{-\lr{\alpha,\beta}}X^{\alpha+\beta},\]
where $\lr{-,-}$ is the Euler form on $K(\mathcal {E})$. According to \cite{Reineke}, we recall the integration map on the Hall algebra $\H(\mathcal {E})$ as follows:
\begin{proposition}\label{jfys}
The integration map $$\int:\H(\mathcal {E})\longrightarrow\mathcal{T}_q(\mathcal {E}),~~[M]\mapsto X^{\hat{M}}$$
is a homomorphism of algebras.
\end{proposition}
\begin{proof}
For reader's convenience, we give the proof here. Given objects $M,N\in\mathcal {E}$,
\begin{flalign*}\int[M]\diamond[N] &= \sum\limits_{[L]} {\frac{{|\Ext_\mathcal{E}^1{{(M,N)}_L}|}}{{|\Hom_\mathcal{E}(M,N)|}}} X^{\hat{L}}\\
&=\sum\limits_{[L]} {\frac{{|\Ext_\mathcal{E}^1{{(M,N)}_L}|}}{{|\Hom_\mathcal{E}(M,N)|}}} X^{\hat{M}+\hat{N}}\\
&=q^{-\lr{M,N}}X^{\hat{M}+\hat{N}}\\
&=X^{\hat{M}}\ast X^{\hat{N}}\\
&=\int[M]\ast\int[N].\end{flalign*}
\end{proof}
Since the global dimension of $C^2(\P)$ is equal to one, we can apply the integration map in Proposition \ref{jfys} to the Hall algebra of $C^2(\P)$. Before doing this, we first give a characterization on the Grothendieck group $K(C^2(\P))$ of $C^2(\P)$. By \cite[Proposition 3.2]{Bau}, for each object $M_\bullet=(\xymatrix{M_1\ar[r]^-{d_1}&M_2})\in C^2(\P)$, we have the following injective resolution
$$0\longrightarrow M_{\bullet}\longrightarrow S_{M_{1}}\oplus J_{M_2}\longrightarrow S_{M_2}\longrightarrow0.$$

Let $P_1,P_2,\cdots,P_n$ be all indecomposable projective objects in $\A$ up to isomorphism.
Fixing a minimal injective resolution of $M_\bullet$
\begin{equation}\label{jxns}0\longrightarrow M_{\bullet}\longrightarrow \bigoplus\limits_{i=1}^na_iS_{P_i}\oplus\bigoplus\limits_{i=1}^nc_iJ_{P_i}\longrightarrow \bigoplus\limits_{i=1}^nb_iS_{P_i}\longrightarrow0,\end{equation} we define the dimension vector $\Dim$ on objects in $C^2(\P)$ by setting $$\Dim M_\bullet=(b_1-a_1,\cdots,b_n-a_n,c_1,\cdots,c_n).$$ By the dual of Lemma \ref{jx}, we obtain that $\Dim M_\bullet$ does not depend on the minimality of injective resolutions. Thus, by the dual of Horseshoe Lemma (cf. \cite[Theorem 12.8]{Bu}), we get the additivity of $\Dim$. That is, for any short exact sequence $$0\longrightarrow M_\bullet \longrightarrow L_\bullet\longrightarrow N_\bullet\longrightarrow0$$ in $C^2(\P)$, we have that $\Dim L_\bullet=\Dim M_\bullet+\Dim N_\bullet$.

\begin{lemma}\label{ftg}
The Grothendieck group $K(C^2(\P))$ is a free abelian group having as a basis the set
$$\{\hat{J}_{P_i},\hat{S}_{P_i}~|~1\leq i\leq n\}$$
and there exists a unique group isomorphism $f: K(C^2(\P))\longrightarrow \mathbb{Z}^{2n}$ such that $f(\hat{M_\bullet})=\Dim M_\bullet$ for each object $M_\bullet$ in $C^2(\P)$.
\end{lemma}
\begin{proof}
For any object $M_\bullet\in C^2(\P)$, taking an injective resolution of $M_\bullet$ as (\ref{jxns}), we obtain that in $K(C^2(\P))$
$$\hat{M}_\bullet=\sum\limits_{i=1}^n((a_i-b_i)\hat{S}_{P_i}+c_i\hat{J}_{P_i}).$$ This shows that $\{\hat{J}_{P_i},\hat{S}_{P_i}~|~1\leq i\leq n\}$ generates the group $K(C^2(\P))$.

For any objects $M_\bullet,N_\bullet\in C^2(\P)$, it is clear that $M_\bullet\cong N_\bullet$ implies $\Dim M_\bullet=\Dim N_\bullet$, since their minimal injective resolutions are isomorphic. Thus, the additivity of $\Dim$ implies the existence of a unique group homomorphism $f: K(C^2(\P))\longrightarrow \mathbb{Z}^{2n}$ such that $f(\hat{M_\bullet})=\Dim M_\bullet$ for each object $M_\bullet$ in $C^2(\P)$. Since the image of the generating set $\{\hat{J}_{P_i},\hat{S}_{P_i}~|~1\leq i\leq n\}$ under the homomorphism $f$ is the canonical basis of the free group $\mathbb{Z}^{2n}$, this set is $\mathbb{Z}$-linearly independent in $K(C^2(\P))$. It follows that $K(C^2(\P))$ is free and that $f$ is an isomorphism.
\end{proof}

Let $\Lambda$ be the bilinear form on $\mathbb{Z}^{2n}$ obtained from the Euler form of $C^2(\P)$ by the isomorphism $f$ in Lemma \ref{ftg}. Define the \emph{quantum torus} associated to the pair $\{\mathbb{Z}^{2n},\Lambda\}$ to be the $\mathbb{Z}[q^{\pm}]$-algebra $\mathcal{T}$ with a basis $\{X^{\bf e}~|~{\bf e}\in\mathbb{Z}^{2n}\}$ and the
multiplication given by
\[X^{\bf e}\ast X^{\bf f}=q^{-\Lambda({\bf e,\bf f})}X^{\bf e+\bf f}.\]
\begin{corollary}\label{jfyscl}
The integration map $$\int:\H(C^2(\P))\longrightarrow\mathcal{T},~~[M_\bullet]\mapsto X^{\Dim{M_\bullet}}$$
is a homomorphism of algebras.
\end{corollary}
\begin{remark}({\bf Further directions})
$(1)$~Since the global dimension of $C^2(\P)$ is equal to one, we want to know whether the Hall algebra $\M\H_2(\A)$ has a bialgebra structure, or say, whether Green's formula on the Hall numbers of a hereditary abelian category holds in $C^2(\P)$;

$(2)$~Let $Q$ be an acyclic quiver of $n$ vertices, and take $\A$ to be the category of finite dimensional $kQ$-modules. Whether we can use the integration map in Corollary \ref{jfyscl} to relate the Hall algebra $\M\H_2(\A)$ with quantum cluster algebras (with principal coefficients). Namely, Whether we can relate the integration map in Corollary \ref{jfyscl} with the work in \cite{DXZ}.
\end{remark}

\section*{Acknowledgments}
The author is grateful to Shiquan Ruan for his suggestion to relate Hall algebra of $m$-cyclic complexes with Hall algebra of $m$-term complexes. He also would like to thank Fan Xu for the conversations on integration maps.

\end{document}